\documentclass[a4paper,11pt,twoside]{amsart}
\usepackage{amsmath}
\usepackage{geometry}
\usepackage{amssymb, latexsym}
\usepackage{verbatim}
\usepackage{enumerate}
\usepackage{comment}
\usepackage{txfonts}
\newtheorem{theorem}{Theorem}[section]
\newtheorem{proposition}[theorem]{Proposition}

\newtheorem{lemma}[theorem]{Lemma}

\newtheorem{definition}[theorem]{Definition}
\newtheorem{remark}[theorem]{Remark}

\newcommand\E{\mathbb{E}}
\newcommand\R{\mathbb{R}}
\newcommand\Z{\mathbb{Z}}
\newcommand\N{\mathbb{N}}

\newcommand\C{\mathbb{C}}

\newcommand\eps{\varepsilon}

\begin{document}
\title[Odd order cases of the logarithmically averaged Chowla conjecture]{Odd order cases of the logarithmically averaged Chowla conjecture}

\author{Terence Tao}
\address{Department of Mathematics, UCLA\\
405 Hilgard Ave\\
Los Angeles CA 90095\\
USA}
\email{tao@math.ucla.edu}

\author{Joni Ter\"av\"ainen}
\address{Department of Mathematics and Statistics, University of Turku\\
20014 Turku\\
Finland}
\email{joni.p.teravainen@utu.fi}

\begin{abstract}
A famous conjecture of Chowla states that the Liouville function $\lambda(n)$ has negligible correlations with its shifts. Recently, the authors established a weak form of the logarithmically averaged Elliott conjecture on correlations of multiplicative functions, which in turn implied all the odd order cases of the logarithmically averaged Chowla conjecture. In this note, we give a new and shorter proof of the odd order cases of the logarithmically averaged Chowla conjecture. In particular, this proof avoids all mention of ergodic theory, which had an important role in the previous proof. 
\end{abstract}

\maketitle

\section{Introduction}

Let $\lambda(n)$ be the Liouville function, defined as $\lambda(n) \coloneqq (-1)^{\Omega(n)}$, with $\Omega(n)$ being the number of prime factors of the integer $n$ counting multiplicity. The distribution of $\lambda(n)$ has been extensively studied.  For instance, the statement
\begin{align*}
\frac{1}{x}\sum_{n\leq x}\lambda(an+b)=o_{x \to \infty}(1)
\end{align*}
for any fixed $a\in \mathbb{N}$, $b\in \mathbb{Z}$ is equivalent to the prime number theorem in arithmetic progressions by an elementary argument. It was conjectured by Chowla \cite{chowla} that we have the significantly more general correlation estimate
\begin{align}\label{eq9}
\frac{1}{x}\sum_{n\leq x}\lambda(a_1n+b_1)\cdots \lambda(a_kn+b_k)=o_{x \to \infty}(1)    
\end{align}
for any $k \geq 1$, $a_1,\dots,a_k,b_1,\dots,b_k\in \mathbb{N}$ satisfying the non-degeneracy condition $a_ib_j-a_jb_i\neq 0$ for $1\leq i<j\leq k$. The non-degeneracy condition may be omitted when $k$ is odd, since a degenerate pair $\lambda(a_i n + b_i) \lambda(a_j n + b_j)$ with $a_i b_j - a_j b_i = 0$ is constant in $n$ and can therefore be deleted. One can of course extend this conjecture to the case when the $b_1,\dots,b_k$ are integers rather than natural numbers (after defining $\lambda$ arbitrarily on negative numbers), but this does of course leads to an equivalent conjecture after applying a translation in the $n$ variable.\\

Chowla's conjecture \eqref{eq9} can be thought of as a simpler analogue of the famous Hardy-Littlewood prime $k$-tuple conjecture \cite{hl}, \cite[Section 1]{gt-linear}, which predicts an asymptotic for the correlations of the von Mangoldt function $\Lambda(n)$. Any rigorous implication between \eqref{eq9} and the Hardy-Littlewood $k$-tuples conjecture, however, would require good savings of the type $O((\log x)^{-A})$ for the error term $o_{x \to \infty}(1)$ in \eqref{eq9} and a large regime of uniformity in the parameters $a_1,\dots,a_k$, $b_1,\dots,b_k$; none of the currently known partial progress on Chowla's conjecture for $k>1$ fulfills these additional requirements. Nevertheless, Chowla's conjecture is subject to the well-known \emph{parity problem} of sieve theory, which also obstructs sieve theoretic approaches to the Hardy-Littlewood prime $k$-tuple conjecture. The parity problem states the fact, first observed by Selberg (see \cite[Chapter 16]{fi}), that classical combinatorial sieves are unable to distinguish numbers with an odd and even number of prime factors from each other.\\

One can also view Chowla's conjecture as a special case of Elliott's conjecture on correlations of multiplicative functions (see \cite[Section 1]{tt} for a modern version of this conjecture, avoiding a technical counterexample to the original conjecture in \cite{elliott}).\\

In \cite{mrt}, Matom\"aki, Radiwi\l{}\l{} and the first author showed that Chowla's conjecture holds on average over the shifts $b_1,\ldots, b_k$, and this was generalised  by Frantzikinakis \cite{frantz} to averages over independent polynomials. Nevertheless, not much is known in the case of individual shifts, unless one considers the logarithmically averaged\footnote{If $1\leq \omega(x)\leq x$ is any function tending to infinity, one could equally well consider \eqref{eq11} with a sum over $\frac{x}{\omega(x)}\leq n\leq x$, with the $\log x$ normalisation replaced by $\log \omega(x)$. In fact, this is what is done in \cite{tao}, \cite{tt}.} version of the conjecture, which states that
\begin{align}\label{eq11}
 \frac{1}{\log x}\sum_{n\leq x}\frac{\lambda(a_1n+b_1)\cdots \lambda(a_kn+b_k)}{n}=o_{x \to \infty}(1),   
\end{align}
provided again that $a_ib_j-a_jb_i\neq 0$ for $1\leq i<j\leq k$. These logarithmically averaged correlations are certainly easier, since \eqref{eq9} implies \eqref{eq11} by partial summation. For the logarithmically averaged variant \eqref{eq11} of Chowla's conjecture, it was shown by the first author \cite{tao} that \eqref{eq11} is $o_{x\to \infty}(1)$ for $k=2$, and we recently showed in \cite{tt} that the same conclusion holds for all odd $k$. Both of these works actually handle more general correlations of bounded multiplicative functions, with \cite{tao} having the same assumptions as in Elliott's conjecture, and \cite{tt} having a non-pretentious assumption for the product of the multiplicative functions (see \cite[Corollary 1.4]{tt} for a precise statement). In addition, it was recently shown by Frantzikinakis and Host \cite[Theorem 1.4]{fh2} that if one replaces the weight $\frac{1}{n}$ in \eqref{eq11} with  $\frac{e^{2\pi i \alpha n}}{n}$ for any irrational $\alpha$, then the analogue of \eqref{eq11} holds for all  $k$. When it comes to conditional results, Frantzikinakis \cite{frantz_ergodic} showed that the logarithmically averaged Chowla conjecture would follow from ergodicity of the measure preserving system associated with the Liouville function.\\  

The proof in \cite{tt} of the odd order cases of the logarithmically averaged Chowla conjecture relies on deep results of Leibman \cite{leibman} and Le \cite{le} on ergodic theory, and is not much simpler than the proof of the structural theorem for correlations of general bounded multiplicative functions in that paper. Here we give a different, shorter proof of the odd order cases of Chowla's conjecture, which avoids all use of ergodic theory, although it now requires the Gowers uniformity of the von Mangoldt function, established by Green, the first author and Ziegler \cite{gt-linear}, \cite{gt-mobius}, \cite{gtz}.  More precisely, we will prove the following.

\begin{theorem}[Odd order cases of the logarithmic Chowla conjecture]\label{theo_logchowla} Let $k\geq 1$ be an odd natural number, and let $a_1,\dots,a_k,b_1,\dots,b_k$ be natural numbers.  Then we have
\begin{align*}
\frac{1}{\log x}\sum_{n\leq x}\frac{\lambda(a_1n+b_1)\cdots \lambda(a_k n+b_k)}{n}=o_{x \to \infty}(1).
\end{align*}
\end{theorem}

\begin{remark} As remarked previously, as we are dealing with an odd number of shifts of the Liouville function, there is no need to impose any non-degeneracy assumptions on the coefficients $a_1,\dots,a_k,b_1,\dots,b_k$.
\end{remark}

\begin{remark} Using the same proof as for Theorem \ref{theo_logchowla}, one could establish an analogous statement for the M\"obius function $\mu(\cdot)$, namely that
\begin{align*}
\frac{1}{\log x}\sum_{n\leq x}\frac{\mu(a_1n+b_1)^{c_1}\cdots \mu(a_kn+b_k)^{c_k}}{n}=o_{x \to \infty}(1)    
\end{align*}
whenever $c_j\geq 1$ are fixed integers with $c_1+\dots +c_k$ odd and $a_j,b_j$ are as above (see also \cite[Corollary 1.6]{tt}).
\end{remark}

\begin{remark} From the proof of Theorem \ref{theo_logchowla}, we see that for the three-point case $k=3$ of Theorem \ref{theo_logchowla}, we only need $U^3$-uniformity of the von Mangoldt function, which was established in \cite{gt-u3} and is simpler than the general $U^k$-uniformity result. In contrast, in \cite{tt} the $k=3$ case was no easier than the general case.
\end{remark}

It was shown by the first author in \cite{tao-higher} that the logarithmically averaged Chowla conjecture \eqref{eq11} for all $k$ is equivalent to two difficult conjectures, namely the logarithmically averaged Sarnak conjecture \cite[Conjecture 1.5]{tao-higher} and the (logarithmic) local Gowers uniformity of the Liouville function \cite[Conjecture 1.6]{tao-higher}. We manage to avoid these problems, since we will only be dealing with odd values of $k$. Indeed, it is natural that the even order cases of Chowla's conjecture are harder than the odd order ones, since one can use the K\'atai-Bourgain-Sarnak-Ziegler orthogonality criterion \cite{katai}, \cite{bsz} to show that the even order cases imply the odd order ones (see \cite[Remark 1.7]{tt}).  We also remark that the proof of Theorem \ref{theo_logchowla} does not require the Matom\"aki-Radziwi{\l}{\l} theorem \cite{mr}, in contrast to the $k=2$ result in \cite{tao} which relied crucially on this theorem.

\subsection{Acknowledgments}

TT was supported by a Simons Investigator grant, the James and Carol Collins Chair, the Mathematical Analysis \&
Application Research Fund Endowment, and by NSF grant DMS-1266164.

JT was supported by UTUGS Graduate School and project number 293876 of the Academy of Finland.\\

Part of this paper was written while the authors were in residence at MSRI in spring 2017, which is supported by NSF grant DMS-1440140.  We thank Kaisa Matom\"aki for helpful discussions and encouragement and Maksym Radziwi{\l}{\l} for suggesting the use of semiprimes in the entropy decrement argument.

\section{Notation} We use standard notation for arithmetic functions throughout this paper. In particular, $\lambda(n)$ is the Liouville function, $\mu(n)$ is the M\"obius function, $\Lambda(n)$ is the von Mangoldt function, and $\varphi(n)$ is the Euler totient function. Various letters, such as $m, n,d,a_j,b_j$, are reserved for integer variables. We use $(n,m)$ to denote the greatest common divisor of $n$ and $m$.  The variable $p$ in turn will always be a prime; in particular, summations such as $\sum_{p \in A} f(p)$ will always be understood to restricted to primes. We will use the standard Landau asymptotic notations $O(\cdot)$, $o(\cdot)$, with $o_{\eta\to 0}(1)$ for instance signifying a quantity that tends to $0$ as $\eta\to 0$; we also use the Vinogradov notation $X \ll Y$ for $X = O(Y)$.\\

For a proposition $P(n)$ depending on $n$, we denote by $1_{P(n)}$ the function that takes value $1$ if $P(n)$ is true and $0$ if it is false. We also use the expectation notations
\begin{align*}
\E_{n\in A}f(n)\coloneqq \frac{\sum_{n\in A}f(n)}{\sum_{n \in A} 1}
\end{align*}
and
\begin{align*}
\E_{n\in A}^{\log} f(n)\coloneqq \frac{\sum_{n\in A} \frac{f(n)}{n}}{\sum_{n \in A} \frac{1}{n}}
\end{align*}
whenever $A$ is a finite non-empty set and $f:A\to \mathbb{C}$ is a function.  If we replace the symbol $n$ by $p$, it is understood that all sums involved are over primes, thus for instance
$$ \E_{p\in A}^{\log} f(p)\coloneqq \frac{\sum_{p\in A} \frac{f(p)}{p}}{\sum_{p \in A} \frac{1}{p}}.$$
Strictly speaking, this average may be undefined if $A$ contains no primes, but in practice we will always be in a regime in which $A$ contains plenty of primes.

\section{The two key subtheorems}

Let $k$ be a natural number, and let $a_1,\dots,a_k,b_1,\dots,b_k$ be natural numbers.  All implied constants in asymptotic notation (and in assertions such as ``$X$ is sufficiently large depending on $Y$'' are henceforth allowed to depend on these quantities.  For any natural number $a$ and any $x \geq 1$, define the quantity
\begin{equation}\label{fax}
 f_x(a) \coloneqq \E_{n \leq x}^{\log} \lambda(a_1n+ab_1)\cdots \lambda(a_k n+ab_k).
\end{equation}
To prove Theorem \ref{theo_logchowla}, it will suffice to show that
\begin{equation}\label{fx0}
f_x(1) \ll \eps
\end{equation}
whenever $\eps>0$, $k$ is odd, and $x$ is sufficiently large depending on $\eps$ (and, by the preceding convention, on $k,a_1,\dots,a_k,b_1,\dots,b_k$).

To obtain \eqref{fx0}, we will rely crucially on the following approximate functional equation for $f_x$, which informally asserts that $f_x(ap) \approx (-1)^k f_x(a)$ for ``most'' $a$ and $p$:

\begin{theorem}[Approximate functional equation]\label{afe}  Let $k,a_1,\dots,a_k,b_1,\dots,b_k$ be natural numbers.  For any $0 < \eps < 1$, $x > 1$, and any natural number $a$, one has
\begin{equation}\label{mpm}
\E_{2^m< p \leq 2^{m+1}} |f_x(ap) - (-1)^k f_x(a)| \ll \eps 
\end{equation}
for all natural numbers $m \leq \log\log x$ outside of an exceptional set ${\mathcal M}$ with
\begin{equation}\label{mam}
 \sum_{m \in {\mathcal M}} \frac{1}{m} \ll a \eps^{-3},
\end{equation}
where the quantity $f_x(a)$ is defined in \eqref{fax}.
\end{theorem}

Results similar to these appear in  \cite[Theorem 3.6]{fh2}, \cite[Theorem 3.6]{tt}.  As in these references, we will prove Theorem \ref{afe} in  Section \ref{afe-proof} via the entropy decrement argument introduced in \cite{tao}; we will use the modification of that argument in \cite{tt} to obtain the relatively strong bound \eqref{mam}.

From \eqref{mpm} we have
$$\E^{\log}_{2^m< p \leq 2^{m+1}} |f_x(ap) - (-1)^k f_x(a)| \ll \eps $$
(since $1/p$ is comparable to $1/2^m$ in the range $2^m < p \leq 2^{m+1}$), and hence from the triangle inequality we have
$$ f_x(a) = (-1)^k \E^{\log}_{2^m < p \leq 2^{m+1}} f_x(ap) + O(\eps) $$
for all $m$ with $2^m \leq (\log x)^{1/2}$ outside of the exceptional set ${\mathcal M}$.  The fact that the average on the right-hand side is over primes will be inconvenient for our argument.  To overcome this, we will establish the following comparison.

\begin{theorem}[Comparison]\label{compar}  Let $k,a_1,\dots,a_k,b_1,\dots,b_k$ be natural numbers.  Let $0 < \eps < 1$, and let
$$ 1 < w < H_- < H_+ < x $$
be parameters with $w$ be sufficiently large depending on $\eps$; $H_-$ sufficiently large depending on $w,\eps$; $H_+$ sufficiently large depending on $H_-, w, \eps$; and $x$ sufficiently large depending on $H_+, H_-, w, \eps$.  Set $W \coloneqq \prod_{p \leq w} p$.  Then, for any natural number $a \leq H_+$ and any $m$ with $H_- \leq 2^m \leq H_+$, one has
$$ \E^{\log}_{2^m < p \leq 2^{m+1}} f_x(ap) = \E^{\log}_{2^m < n \leq 2^{m+1}: (n,W)=1} f_x(an) + O(\eps)$$
where the quantity $f_x(a)$ is defined in \eqref{fax}.
\end{theorem}

We will prove this assertion in Section \ref{compar-proof}.  Our main tool will be the theory of the Gowers uniformity norms, and in particular the Gowers uniformity of the $W$-tricked von Mangoldt function proven in \cite{gt-linear}, \cite{gt-mobius}, \cite{gtz}.  In contrast to Theorem \ref{afe}, the bounds in Theorem \ref{compar} (particularly with regards to what ``sufficiently large'' means) are qualitative rather than quantitative; this is primarily due to the qualitative nature of the bounds currently available for the Gowers uniformity of the $W$-tricked von Mangoldt function.  A key technical point in the above theorem is that the parameter $a$ is permitted to be large compared to the parameter $w$ (or $W$); this will be important in the argument below.

In the remainder of this section we show how Theorem \ref{afe} and Theorem \ref{compar} yield \eqref{fx0} when $k$ is odd and $x$ is sufficiently large depending on $\eps$.

Fix $0 < \eps < 1/2$.  We will need parameters
\begin{equation}\label{params}
 \frac{1}{\eps} < w < H_1 < H_2 < H_3 < H_4 < x 
\end{equation}
with $w$ sufficiently large depending on $\eps$, each $H_i$ for $i=1,2,3,4$ sufficiently large depending on $w,\eps$ and $H_1,\dots,H_{i-1}$, and $x$ sufficiently large depending on $H_4,H_3,H_2,H_1,w,\eps$.

From Theorem \ref{afe} and the hypothesis that $k$ is odd, one has
$$ f_x(1) = -\E^{\log}_{2^m < p_1 \leq 2^{m+1}} f_x(p_1) + O(\eps)$$
for all $m$ in the range $H_1 \leq 2^m \leq H_2$, outside of an exceptional set ${\mathcal M}_1$ with
$$ \sum_{m \in {\mathcal M}_1} \frac{1}{m} \ll \eps^{-3}.$$
For $m$ in this exceptional set, we of course have
$$ f_x(1) = -\E^{\log}_{2^m < p_1 \leq 2^{m+1}} f_x(p_1) + O(1).$$
Averaging over all such $m$ and using the prime number theorem, we conclude (given the hypotheses on the parameters \eqref{params}) that
\begin{equation}\label{fx1}
 f_x(1) = -\E^{\log}_{H_1 < p_1 \leq H_2} f_x(p_1) + O(\eps).
\end{equation}
A similar application of Theorem \ref{afe} yields
\begin{equation}\label{fx2}
 f_x(1) = -\E^{\log}_{H_3 < p \leq H_4} f_x(p) + O(\eps).
\end{equation}
Also, applying Theorem \ref{afe} with $a$ replaced by $p_1$, we have
$$ f_x(p_1) = -\E^{\log}_{H_3 < p_2 \leq H_4} f_x(p_1 p_2) + O(\eps)$$
for all primes $p_1$ with $H_1 < p_1 \leq H_2$; inserting this into \eqref{fx1}, we obtain
\begin{equation}\label{fx3}
 f_x(1) = +\E^{\log}_{H_1 < p_1 \leq H_2} \E^{\log}_{H_3 < p_2 \leq H_4} f_x(p_1 p_2) + O(\eps).
\end{equation}
Crucially, the sign in \eqref{fx3} is the opposite of the sign in \eqref{fx2}.  To conclude the proof of \eqref{fx0} from \eqref{fx2}, \eqref{fx3}, it will suffice to show that the average \eqref{fx2} involving primes $p$ and the average \eqref{fx3} involving semiprimes $p_1 p_2$ are comparable in the sense that
\begin{equation}\label{fx4}
\E^{\log}_{H_3 < p \leq H_4} f_x(p) = \E^{\log}_{H_1 < p_1 \leq H_2} \E^{\log}_{H_3 < p_2 \leq H_4} f_x(p_1 p_2) + O(\eps).
\end{equation}
To do this, we use Theorem \ref{compar} several times.  Firstly, from this theorem we see that
$$
 \E^{\log}_{2^m < p \leq 2^{m+1}} f_x(p) =  \E^{\log}_{2^m < n \leq 2^{m+1}: (n,W) = 1} f_x(n) + O(\eps)
$$
whenever $H_3 \leq 2^m \leq H_4$; averaging over $m$ (and noting that the error terms that arise can be easily absorbed into the $O(\eps)$ error) we conclude that
$$ \E^{\log}_{H_3 < p \leq H_4} f_x(p) = \E^{\log}_{H_3 < n \leq H_4: (n,W)=1} f_x(n) + O(\eps).$$
Similarly, we have
$$ \E^{\log}_{H_3 < p_2 \leq H_4} f_x(p_1 p_2)  = \E^{\log}_{H_3 < n_2 \leq H_4: (n_2,W)=1} f_x(p_1 n_2) + O(\eps)$$
whenever $H_1 <p_1 \leq H_2$ (note that this is despite $p_1$ being large compared with $w$ or $W$).
Thus it will suffice to show that
\begin{equation}\label{fx5}
\E^{\log}_{H_3 < n \leq H_4: (n,W)=1} f_x(n) = \E^{\log}_{H_1 < p_1 \leq H_2} \E^{\log}_{H_3 < n_2 \leq H_4: (n_2,W)=1} f_x(p_1 n_2) + O(\eps).
\end{equation}
By making the change of variables $n = p_1 n_2$, and noting that $n$ is coprime to $W$ if and only if $n_2$ is, we can write
$$
\E^{\log}_{H_3 < n_2 \leq H_4: (n_2,W)=1} f_x(p_1 n_2) = \E^{\log}_{p_1 H_3 < n \leq p_1 H_4: (n,W)=1} f_x(n) p_1 1_{p_1|n} + O(\eps),$$
and one can modify the range $p_1 H_3 < n \leq p_1 H_4$ to $H_3 < n \leq H_4$ incurring a further error of $O(\eps)$.  We may thus rearrange \eqref{fx5} as
$$
\E^{\log}_{H_3 < n \leq H_4: (n,W)=1} f_x(n) (g(n)-1) = O(\eps)$$
where $g$ is the weight
$$ g(n) \coloneqq \E^{\log}_{H_1 < p_1 \leq H_2} p_1 1_{p_1|n}.$$
By the Cauchy-Schwarz inequality and the boundedness of $f_x$, it thus suffices to establish the bound
$$
\E^{\log}_{H_3 < n \leq H_4: (n,W)=1} (g(n)-1)^2 \ll \eps^2$$
which will follow in turn from the bounds
\begin{equation}\label{g1}
\E^{\log}_{H_3 < n \leq H_4: (n,W)=1} g(n) = 1 + O(\eps^2)
\end{equation}
and
\begin{equation}\label{g2}
\E^{\log}_{H_3 < n \leq H_4: (n,W)=1} g(n)^2 = 1 + O(\eps^2).
\end{equation}
The left-hand side of \eqref{g1} can be rewritten as
$$ \E^{\log}_{H_1 < p_1 \leq H_2} p_1 \E^{\log}_{H_3 < n \leq H_4: (n,W)=1} 1_{p_1|n} $$
and the claim \eqref{g1} follows since one can easily compute that
$$ \E^{\log}_{H_3 < n \leq H_4: (n,W)=1} 1_{p_1|n} = \frac{1 + O(\eps^2)}{p_1}.$$
Similarly, the left-hand side of \eqref{g2} can be rewritten as
$$ \E^{\log}_{H_1 < p_1 \leq H_2} \E^{\log}_{H_1 < p'_1 \leq H_2}  p_1 p'_1 \E^{\log}_{H_3 < n \leq H_4: (n,W)=1} 1_{p_1,p'_1|n} $$
and the claim \eqref{g2} follows since $\E^{\log}_{H_3 < n \leq H_4: (n,W)=1} 1_{p_1,p'_1|n}$ is equal to $\frac{1+O(\eps^2)}{p_1 p'_1}$ when $p_1 \neq p'_1$, and can be bounded crudely by $O(1/p_1)$ when $p_1=p'_1$.  This concludes the proof of Theorem \ref{theo_logchowla}, except for the proofs of Theorem \ref{afe} and Theorem \ref{compar} which will be accomplished in the next two sections respectively.

\section{Using the entropy decrement argument}\label{afe-proof}

We now prove Theorem \ref{afe}.  Let $k,a_1,\dots,a_k,b_1,\dots,b_k,\eps,a,x$ be as in that theorem.  We may assume that
\begin{equation}\label{xam}
 x \geq \exp\exp\exp( a \eps^{-3} )
\end{equation}
since otherwise the claim is trivial by setting ${\mathcal M}$ to consist of all $m \leq \log \log x$.  We may also restrict attention to proving
\eqref{mpm} for $m$ satisfying
\begin{equation}\label{msm}
\exp(a \eps^{-3}) \leq m \leq \frac{1}{100} \log\log x
\end{equation}
since all the $m$ between $\frac{1}{100} \log\log x$ and $\log\log x$, or less than $\exp(a \eps^{-3})$, can be placed in the exceptional set ${\mathcal M}$ without significantly affecting \eqref{mam}.  Finally, we can assume that $\eps \leq 1/2$, since for $1/2 < \eps \leq 1$ the bound \eqref{mpm} holds from the triangle inequality.

For any prime $p$, one has the identity
$$ \lambda(n) = - \lambda(pn)$$
for any natural number $n$, and hence
$$ \lambda(a_1n+ab_1)\cdots \lambda(a_k n+ab_k) = (-1)^k \lambda(a_1pn+apb_1)\cdots \lambda(a_k n+apb_k).$$
From \eqref{fax} we thus have
$$
 f_x(a) = (-1)^k \E_{n \leq x}^{\log} \lambda(a_1pn+apb_1)\cdots \lambda(a_kpn+apb_k).$$
If $p \leq \log x$, then (using \eqref{xam}) we have $\sum_{x < n \leq px} \frac{1}{n} \ll \eps \sum_{n \leq x} \frac{1}{n}$, and hence\footnote{Here it is essential that we are using logarithmic averaging; the argument breaks down completely at this point if one uses ordinary averaging.} that
$$ \E_{n \leq px}^{\log} g(n) = \E_{n \leq x}^{\log} g(n) + O(\eps)$$
whenever $g: {\mathbb N} \to \C$ is bounded in magnitude by $1$.  Thus we have
$$ f_x(a) = (-1)^k \E_{n \leq px}^{\log} \lambda(a_1pn+apb_1)\cdots \lambda(a_kpn+apb_k) + O(\eps)$$
for all $p\leq \log x$. Making the change of variables $n' \coloneqq pn$, we conclude that
$$
 f_x(a) = (-1)^k \E_{n' \leq x}^{\log} \lambda(a_1n'+apb_1)\cdots \lambda(a_kn'+apb_k) p 1_{p|n'} + O(\eps).$$
 Replacing $n'$ with $n$, and comparing with \eqref{fax} with $a$ replaced by $ap$, we conclude that
$$ f_x(a) - (-1)^k f_x(ap) = (-1)^k \E_{n \leq x}^{\log} \lambda(a_1n+apb_1)\cdots \lambda(a_kn+apb_k) (p 1_{p|n}-1) + O(\eps).$$
The contribution of those $n$ with $n \leq x^\eps$ is $O(\eps)$, so we have
$$ f_x(a) - (-1)^k f_x(ap) = (-1)^k \E_{x^\eps < n \leq x}^{\log} \lambda(a_1n+apb_1)\cdots \lambda(a_kn+apb_k) (p 1_{p|n}-1) + O(\eps)$$
for all $p\leq \log x$. If we set $c_p \in \{-1,0,+1\}$ to be the signum of $\E_{n \leq x}^{\log} \lambda(a_1n+apb_1)\cdots \lambda(a_kn+apb_k) (p 1_{p|n}-1)$, it will thus suffice to show that
\begin{equation}\label{sa}
\E_{2^m < p \leq 2^{m+1}} c_p \E_{x^\eps < n \leq x}^{\log} \lambda(a_1n+apb_1)\cdots \lambda(a_kn+apb_k) (p 1_{p|n}-1) = O(\eps) 
\end{equation}
for all $m$ obeying \eqref{msm}, outside of an exceptional set ${\mathcal M}$ obeying \eqref{mam}.

Let $m$ obey \eqref{msm}. If $j$ is a natural number less than or equal to $2^m$ (and hence of size $O( \log^{1/10} x )$), one easily computes the total variation bound
$$ \sum_{x^\eps < n \leq x^\eps + j} \frac{1}{n} + \sum_{x^\eps + j < n \leq x+j} \left|\frac{1}{n} - \frac{1}{n+j}\right| \ll \frac{\log^{1/10} x}{x^\eps} $$
and thus
$$ \E_{x^{\eps}\leq n \leq x}^{\log} g(n) = \E_{x^{\eps}\leq n \leq x}^{\log} g(n+j) + O\left(\frac{\log^{1/10} x}{x^\eps \log x}\right)$$
for any function $g: \N \to \C$ bounded in magnitude by $1$.  By \eqref{xam}, the error term is certainly of size $O(\eps)$.
In particular, the left-hand side of \eqref{sa} can be written as
$$
\E_{2^m \leq p \leq 2^{m+1}} c_p \E_{x^\eps < n \leq x}^{\log} \lambda(a_1n+a_1 j + apb_1)\cdots \lambda(a_kn+a_kj + apb_k) (p 1_{p|n+j}-1) + O(\eps) $$
for any $1 \leq j \leq 2^m$.  Averaging in $j$ and rearranging, we can thus write the left-hand side of \eqref{sa} in probabilistic language\footnote{We will use boldface symbols such as $\mathbf{n}, \mathbf{X}_m, \mathbf{Y}_m, \mathbf{Z}_m$ to denote random variables, with non-boldface symbols such as $X_m$ being used to denote deterministic variables instead.} as
$$
\mathbf{E} \mathbf{Z}_m + O(\eps),$$
where $\mathbf{E}$ denotes expectation, $\mathbf{Z}_m$ is the random variable
$$
\mathbf{Z}_m \coloneqq \E_{2^m < p \leq 2^{m+1}} \E_{j \leq 2^m} c_p \lambda(a_1\mathbf{n}+a_1 j + apb_1)\cdots \lambda(a_k\mathbf{n}+a_kj + apb_k) (p 1_{p|\mathbf{n}+j}-1),$$
and $\mathbf{n}$ is a random natural number in the interval $(x^\eps,x]$ drawn using the logarithmic distribution
$$ \mathbf{P}( \mathbf{n} = n ) = \frac{1/n}{\sum_{x^\eps < n' \leq x} \frac{1}{n'}}$$
for all $x^\eps < n \leq x$.

We now ``factor'' the random variable $\mathbf{Z}_m$ into a function of two other random variables $\mathbf{X}_m, \mathbf{Y}_m$, defined as follows. Let $B:=\max_{i}b_i$ and
$$ C \coloneqq \sum_{i=1}^k (2aB+1) a_i, $$
and let $\mathbf{X}_m \in \{-1,+1\}^{C 2^m}$ and $\mathbf{Y}_m \in \prod_{2^m < p \leq 2^{m+1}} \Z/p\Z$ be the random variables
$$ \mathbf{X}_m \coloneqq ( \lambda( a_i \mathbf{n} + r ) )_{1 \leq i \leq k; 1 \leq r \leq (2aB+1) a_i 2^m} $$
and
$$ \mathbf{Y}_m \coloneqq ( \mathbf{n} \hbox{ mod } p )_{2^m < p \leq 2^{m+1}}.$$
Then we may write $\mathbf{Z}_m = F_m( \mathbf{X}_m, \mathbf{Y}_m )$, where $F_m: \{-1,+1\}^{C 2^m} \times \prod_{2^m < p \leq 2^{m+1}} \Z/p\Z \to \R$ is the function defined by
$$ 
F_m ( (b_{i,r})_{1 \leq i \leq k; 1 \leq r \leq (2aB+1) a_i 2^m}, (n_p)_{2^m < p \leq 2^{m+1}} )
\coloneqq
\E_{2^m < p \leq 2^{m+1}} \E_{j \leq 2^m} c_p b_{1,a_1 j + apb_1} \cdots b_{k,a_k j + apb_k} (p 1_{p|n_p+j}-1).$$
for all $b_{i,r} \in \{-1,+1\}$ and $n_p \in \Z/p\Z$.  It will now suffice to show that
$$ \mathbf{E} F_m( \mathbf{X}_m, \mathbf{Y}_m ) = O(\eps) $$
for all $m$ obeying \eqref{msm}, outside of an exceptional set $\mathcal{M}$ obeying \eqref{mam}.

At this point we recall some information-theoretic concepts:

\begin{definition}[Entropy and conditional expectation]  Let $\mathbf{X}, \mathbf{Y}, \mathbf{Z}$ be random variables taking finitely many values.  Then we have the entropy
$$ \mathbf{H}(\mathbf{X}) \coloneqq \sum_x \mathbf{P}( \mathbf{X} = x ) \log \frac{1}{\mathbf{P}( \mathbf{X} = x )}$$
where the sum is over all $x$ for which $\mathbf{P}(\mathbf{X} = x)\neq 0$.  Similarly we have the conditional entropy
$$ \mathbf{H}(\mathbf{X}|E) \coloneqq \sum_x \mathbf{P}( \mathbf{X} = x |E) \log \frac{1}{\mathbf{P}( \mathbf{X} = x |E)}$$
for any event $E$ of positive probability, and
$$ \mathbf{H}(\mathbf{X}|\mathbf{Y}) \coloneqq \sum_y \mathbf{P}(\mathbf{Y} = y) \mathbf{H}( \mathbf{X} | \mathbf{Y} = y ).$$
Finally, we define the mutual information
$$ \mathbf{I}(\mathbf{X} : \mathbf{Y}) = \mathbf{H}(\mathbf{X}) - \mathbf{H}(\mathbf{X}|\mathbf{Y}) = \mathbf{H}(\mathbf{Y}) - \mathbf{H}(\mathbf{Y}|\mathbf{X}),$$
and similarly define the conditional mutual information
$$ \mathbf{I}(\mathbf{X} : \mathbf{Y}|\mathbf{Z}) = \mathbf{H}(\mathbf{X}|\mathbf{Z}) - \mathbf{H}(\mathbf{X}|\mathbf{Y},\mathbf{Z}) = \mathbf{H}(\mathbf{Y}|\mathbf{Z}) - \mathbf{H}(\mathbf{Y}|\mathbf{X},\mathbf{Z}).$$
\end{definition}

For each $m$ obeying \eqref{msm}, let $\mathbf{Y}_{<m}$ be the random variable $\mathbf{Y}_{<m} \coloneqq (\mathbf{Y}_{m'})_{m' < m}$.
  We can control the expectation $\mathbf{E} F_m( \mathbf{X}_m, \mathbf{Y}_m )$ by the conditional mutual information $\mathbf{I}( \mathbf{X}_m : \mathbf{Y}_m | \mathbf{Y}_{<m} )$ as follows: 

\begin{proposition}  Suppose $m$ obeys \eqref{msm} and is such that
\begin{equation}\label{lun}
 \mathbf{I}( \mathbf{X}_m : \mathbf{Y}_m | \mathbf{Y}_{<m} ) \leq \eps^3 \frac{2^m}{m}.
\end{equation}
Then one has
$$ \mathbf{E} F_m( \mathbf{X}_m, \mathbf{Y}_m ) \ll \eps.$$
\end{proposition}

\begin{proof}  We argue as in \cite{tt}, which are in turn a modification of the arguments in \cite{tao}.  Let $\mathbf{U}_m$ be drawn uniformly at random from $\prod_{2^m < p \leq 2^{m+1}} \Z/p\Z$.  We first show that for any sign pattern $X_m \in \{-1,+1\}^{C 2^m}$, one has
\begin{equation}\label{pa}
 \mathbf{P}( |F_m( X_m, \mathbf{U}_m )| \geq \eps ) \ll \exp( - c \eps^2 2^m / m ) 
\end{equation}
for an absolute constant $c>0$.  If we write $\mathbf{U}_m = (\mathbf{n}_p)_{2^m < p \leq 2^{m+1}}$, then the $\mathbf{n}_p$ are jointly independent in $p$ and uniformly distributed on $\Z/p\Z$.  If $X_m = (b_{i,r})_{1 \leq i \leq k; 1 \leq r \leq (2aB+1)a_i 2^m}$, then one can write
$$ F_m( X_m, \mathbf{U}_m ) = \E_{2^m < p \leq 2^{m+1}} \mathbf{W}_p$$
where $\mathbf{W}_p$ is the random variable
$$ \mathbf{W}_p \coloneqq \E_{j \leq 2^m} c_p b_{1,a_1 j + apb_1} \cdots b_{k,a_k j + apb_k} (p 1_{p|\mathbf{n}_p+j}-1).$$
Observe that the $\mathbf{W}_p$ are jointly independent, bounded in magnitude by $O(1)$, and have mean zero.  The claim \eqref{pa} now follows from Hoeffding's inequality \cite{hoeff}.

Applying the Pinsker-type inequality from \cite[Lemma 3.4]{tt} (see also \cite[Lemma 3.3]{tao}), we conclude that
$$
 \mathbf{P}( |F_m( X_m, \mathbf{Y} )| \geq \eps ) \ll \frac{m}{\eps^2 2^m} ( \mathbf{H}(\mathbf{U}_m) - \mathbf{H}(\mathbf{Y}) + 1 )$$
for any random variable $\mathbf{Y}$ taking values in $(\mathbf{n}_p)_{2^m < p \leq 2^{m+1}}$; in particular, applying this to the probability measure $\mathbf{P}'(E):=\mathbf{P}(E|\mathbf{X}_m = X_m, \mathbf{Y}_{<m} = Y_{<m})$, we have
$$
 \mathbf{P}( |F_m( \mathbf{X}_m, \mathbf{Y}_m )| \geq \eps | \mathbf{X}_m = X_m, \mathbf{Y}_{<m} = Y_{<m} ) \ll \frac{m}{\eps^2 2^m} ( \mathbf{H}(\mathbf{U}_m) - \mathbf{H}(\mathbf{Y}_m|\mathbf{X}_m = X_m, \mathbf{Y}_{<m} = Y_{<m}) + 1 ) .$$
Averaging over $X_m,Y_{<m}$, we conclude that
$$
 \mathbf{P}( |F_m( \mathbf{X}_m, \mathbf{Y}_m )| \geq \eps ) \ll \frac{m}{\eps^2 2^m} ( \mathbf{H}(\mathbf{U}_m) - \mathbf{H}(\mathbf{Y}_m|\mathbf{X}_m, \mathbf{Y}_{<m}) + 1 ) ,$$
and hence (since $F_m$ is bounded by $O(1)$, and $m$ is large compared to $1/\eps$)
$$
 \mathbf{E} |F_m( \mathbf{X}_m, \mathbf{Y}_m )| \ll \frac{m}{\eps^2 2^m} ( \mathbf{H}(\mathbf{U}_m) - \mathbf{H}(\mathbf{Y}_m|\mathbf{X}_m, \mathbf{Y}_{<m}) ) + \eps.$$
We can write
$$\mathbf{H}(\mathbf{Y}_m|\mathbf{X}_m, \mathbf{Y}_{<m}) = \mathbf{H}(\mathbf{Y}_m|\mathbf{Y}_{<m}) - \mathbf{I}( \mathbf{X}_m : \mathbf{Y}_m | \mathbf{Y}_{<m} ) $$
and hence by \eqref{lun} we have
\begin{equation}\label{do}
 \mathbf{E} |F_m( \mathbf{X}_m, \mathbf{Y}_m )| \ll \frac{m}{\eps^2 2^m} ( \mathbf{H}(\mathbf{U}_m) - \mathbf{H}(\mathbf{Y}_m|\mathbf{Y}_{<m}) ) + \eps.
\end{equation}
Uniformly for $1\leq b\leq q\leq x^{\varepsilon}$, we have the simple estimate
$$
\sum_{\substack{x^{\eps}\leq n\leq x\\n\equiv b \pmod{q}}}\frac{1}{n}=\left(\frac{1}{q}+O\left(\frac{q}{x^{\eps}}\right)\right)\sum_{x^{\eps}\leq n\leq x} \frac{1}{n},
$$
so from the Chinese remainder theorem (and the prime number theorem), we see that the random variable $\mathbf{Y}_m$, after conditioning to any event of the form $\mathbf{Y}_{<m} = Y_{<m}$, is almost uniformly distributed in the sense that
\begin{equation}\label{eqq27}\mathbf{P}( \mathbf{Y}_m = Y_m | \mathbf{Y}_{<m} = Y_{<m} ) = \frac{1}{\prod_{2^m < p \leq 2^{m+1}} p} + O\left( \frac{\exp( O(2^m) )}{x^{\eps}} \right).
\end{equation}
We have for any distinct $x,y\in (0,1]$ the elementary inequality\footnote{Assuming by symmetry that $y>x$, and writing $y=x+\delta$ with $\delta=|y-x|$, the inequality follows from the mean value theorem applied to $x\mapsto (x+\delta)\log\frac{1}{x+\delta}$.}
$$
\left|x\log \frac{1}{x}-y\log \frac{1}{y}\right|\leq C|x-y|\log \frac{2}{|x-y|}\leq 2C|x-y|^{\frac{1}{2}}
$$
for some constant $C>0$, so if $\mathbf{X}$ and $\mathbf{X}'$ are any random variables having the same finite range $\mathcal{X}$, then we can compare their entropies by
\begin{equation}\label{eqq28}
|\mathbf{H}(\mathbf{X})-\mathbf{H}(\mathbf{X}')|\leq 2C\cdot \max_{x\in \mathcal{X}}|\mathbf{P}(\mathbf{X}=x)-\mathbf{P}(\mathbf{X}'=x)|^{\frac{1}{2}}\cdot |\mathcal{X}|.
\end{equation}
From this and \eqref{eqq27} we compute that
$$ \mathbf{H}(\mathbf{U}_m) - \mathbf{H}(\mathbf{Y}_m|\mathbf{Y}_{<m}) \ll \frac{\exp( O(2^m) )}{x^{\eps/2}}.$$
Inserting this into \eqref{do} and using \eqref{xam}, \eqref{msm} we conclude that
$$
 \mathbf{E} |F_m( \mathbf{X}_m, \mathbf{Y}_m )| \ll \eps
$$
as required.
\end{proof}

Theorem \ref{afe} now follows from the preceding proposition and the following estimate.

\begin{proposition}[Entropy decrement argument]  One has
$$ \sum_{\exp(a \eps^{-3}) \leq m \leq \frac{1}{100} \log\log x} \frac{1}{2^m} \mathbf{I}( \mathbf{X}_m : \mathbf{Y}_m | \mathbf{Y}_{<m} ) \ll a.$$
\end{proposition}

\begin{proof}  For any $m$ obeying \eqref{msm}, consider the quantity
$$ \mathbf{H}( \mathbf{X}_{m+1} | \mathbf{Y}_{<m+1} ).$$
We can view $\mathbf{X}_{m+1}$ as a pair $(\mathbf{X}_m, \mathbf{X}'_{m})$, where
$$ \mathbf{X}'_{m} \coloneqq ( \lambda( a_i \mathbf{n}' + r ) )_{1 \leq i \leq k; 1 \leq r \leq (2aB+1) a_i 2^m} $$
and $\mathbf{n}' \coloneqq \mathbf{n} + (2aB+1) 2^m$.   By the Shannon entropy inequalities, we thus have
$$ \mathbf{H}( \mathbf{X}_{m+1} | \mathbf{Y}_{<m+1} )
\leq \mathbf{H}( \mathbf{X}_{m} | \mathbf{Y}_{<m+1} ) + \mathbf{H}( \mathbf{X}'_{m} | \mathbf{Y}_{<m+1} ).
$$
If we write
$$ \mathbf{Y}'_{<m+1} \coloneqq ( \mathbf{n}' \hbox{ mod } p )_{p \leq 2^{m+1}}$$
then $\mathbf{Y}_{<m+1}$ and $\mathbf{Y}'_{<m+1}$ define the same $\sigma$-algebra (each random variable is a deterministic function of the other), and so we have
$$ \mathbf{H}( \mathbf{X}_{m+1} | \mathbf{Y}_{<m+1} )
\leq \mathbf{H}( \mathbf{X}_{m} | \mathbf{Y}_{<m+1} ) + \mathbf{H}( \mathbf{X}'_{m} | \mathbf{Y}'_{<m+1} ).
$$
The total variation distance between $\mathbf{n}$ and $\mathbf{n}'$  can be computed to be $O( \exp(O(2^m)) / x^\eps )$.  Since $\mathbf{Y}_{<m+1}$ takes on $O( \exp(O(2^m)))$ values, we see from \eqref{eqq28} that
$$ \mathbf{H}( \mathbf{Y}'_{<m+1} ) = \mathbf{H}(\mathbf{Y}_{<m+1} ) + O( \exp(O(2^m)) / x^{\eps/2} ).$$
Similarly, since the random variables $(\mathbf{X}_{m},\mathbf{Y}_{<m+1})$ and $(\mathbf{X}_{m}',\mathbf{Y}_{<m+1}')$ also  take on $O( \exp(O(2^m)))$ values and are deterministic functions of $\mathbf{n}$ and $\mathbf{n'}$, respectively, by \eqref{eqq28} we again have
$$ \mathbf{H}( \mathbf{X}'_{m}, \mathbf{Y}'_{<m+1} ) = \mathbf{H}(\mathbf{X}_{m}, \mathbf{Y}_{<m+1} ) + O( \exp(O(2^m)) / x^{\eps/2} ),$$
and hence on subtracting
$$ \mathbf{H}( \mathbf{X}'_{m}| \mathbf{Y}'_{<m+1} ) = \mathbf{H}(\mathbf{X}_{m}| \mathbf{Y}_{<m+1} ) + O( \exp(O(2^m)) / x^{\eps/2} ).$$
Thus we have
$$ \mathbf{H}( \mathbf{X}_{m+1} | \mathbf{Y}_{<m+1} )
\leq 2\mathbf{H}( \mathbf{X}_{m} | \mathbf{Y}_{<m+1} ) + O( \exp(O(2^m)) / x^{\eps/2} ).$$
But we can write $\mathbf{Y}_{<m+1}$ as a pair $(\mathbf{Y}_{<m}, \mathbf{Y}_m)$, to conclude that
$$ \mathbf{H}( \mathbf{X}_{m} | \mathbf{Y}_{<m+1} ) = \mathbf{H}( \mathbf{X}_{m} | \mathbf{Y}_{<m} ) - \mathbf{I}( \mathbf{X}_m : \mathbf{Y}_m | \mathbf{Y}_{<m} ).$$
Inserting this identity and rearranging, we conclude that
$$ \frac{1}{2^m} \mathbf{I}( \mathbf{X}_m : \mathbf{Y}_m | \mathbf{Y}_{<m} ) \leq
 \frac{1}{2^m} \mathbf{H}( \mathbf{X}_{m} | \mathbf{Y}_{<m} )-\frac{1}{2^{m+1}} \mathbf{H}( \mathbf{X}_{m+1} | \mathbf{Y}_{<m+1} )  + O( \exp(O(2^m)) / x^{\eps/2} )$$
and thus on summing the telescoping series
$$ \sum_{m \leq \frac{1}{100} \log\log x} \frac{1}{2^m} \mathbf{I}( \mathbf{X}_m : \mathbf{Y}_m | \mathbf{Y}_{<m} ) \ll \mathbf{H}(\mathbf{X}_1) + 1$$
(say).  Since $\mathbf{X}_1$ takes at most $\exp(O(a))$ values, we have $\mathbf{H}(\mathbf{X}_1) = O(a)$, and the claim follows.
\end{proof}

\section{Using the Gowers norms}\label{compar-proof}

We now prove Theorem \ref{compar}.  As stated previously, we will rely heavily on the theory of the Gowers norms, which we now recall.

\begin{definition}[Gowers norms]  Given integers $k\geq 1$ and $N\geq 1$ and a function $f:\Z/N\Z\to \mathbb{C}$, we define the Gowers norms $U^k(\Z/N\Z)$ by
\begin{align*}
 \|f\|_{U^k(\Z/N\Z)}\coloneqq \left(\E_{n\in \Z/N\Z}\E_{h_1,\ldots, h_k\in \Z/N\Z}\prod_{\mathbf{\omega} \in \{0,1\}^k}\mathcal{C}^{|\mathbf{\omega}|}f(n+\mathbf{\omega}\cdot \mathbf{h})\right)^{2^{-k}}, 
\end{align*}
where $\mathcal{C}$ is the complex conjugation operator, $|\omega|$ is the number of ones in $\omega\in \{0,1\}^k$, $\mathbf{h}=(h_1,\ldots, h_k)$, and $\cdot$ denotes the inner product of two vectors. One easily sees that $\|f\|_{U^k(\Z/N\Z)}$ is a well-defined nonnegative quantity. We can then define the Gowers $U^k[N]$-norm of a function $f:\{1,\ldots, N\}\to \mathbb{C}$ defined on a finite interval by 
\begin{align*}
 \|f\|_{U^k[N]}\coloneqq \frac{\left\|f\cdot 1_{[1,N]}\right\|_{U^k(\mathbb{Z}_{N'})}}{\left\|1_{[1,N]}\right\|_{U^k(\mathbb{Z}_{N'})}}   
\end{align*}
where $N'=3N$, say (one easily sees that the definition is independent of the choice of $N'>2N$) and $f\cdot 1_{[1,N]}$ is to be interpreted as a function of period $N'$, and hence as a function on $\mathbb{Z}_{N}'$. 
\end{definition}

For the basic properties of Gowers norms, see \cite[Chapter 11]{tao-vu}.  The main general fact we will need about these norms is the following.

\begin{lemma}[A generalised von Neumann theorem]\label{le1} For $k\in \mathbb{N}$, let $\theta, \phi_1,\ldots, \phi_k:\mathbb{Z}\to \mathbb{C}$ be functions with $|\phi_j|\leq 1$. Also let $a_j,b_j,r_j\in \mathbb{Z}$ for $1\leq j\leq k$, and $W\in \mathbb{N}$ with $W\leq N^{0.1}$. Then  
\begin{align*}
\left|\E_{d\leq \frac{N}{W}}\E_{n\leq N}\theta(d)\phi_1(a_1 n+Wb_1d+r_1)\cdots \phi_k(a_k n+Wb_{k}d+r_k)\right|\leq C \|\theta\|_{U^{k}[\frac{N}{W}]}+o_{N\to \infty}(1)   
\end{align*}
for some constant $C>0$ depending only on $k$ and the numbers $a_1,\dots,a_k,b_1,\dots,b_k$, but independent of $W$ and $r_1,\dots,r_k$.
\end{lemma}

Without the $W$-aspect, this is standard; see for instance \cite[Lemma 2]{fhk}. However, the uniformity of the bounds in $W$ (and $r_1,\dots,r_k$) will be crucial in our arguments.

\begin{proof} We shall adapt the proof of \cite[Proposition 3.3]{tao-higher}. By splitting the variable $n$ into residue classes $\pmod W$ and setting $N'\coloneqq \frac{N}{W}$ it suffices to show that
\begin{align*}
\left|\E_{d\leq N'}\E_{n\leq N'}\theta(d)\phi_1(W(a_1 n+b_1d)+r_1')\cdots \phi_k(W(a_k n+b_{k}d)+r_k')\right|\leq C \|\theta\|_{U^{k}[N']}+o_{N'\to \infty}(1)     
\end{align*}
for all integers $r'_1,\dots,r'_k$. To simplify notation, we will call $N'$ just $N$. By considering the functions $\widetilde{\phi}_j(n)\coloneqq \phi_j(Wn+r_j')$, we see that it suffices to prove for all functions $|\phi_j|\leq 1$ that
\begin{align}\label{eqq26}
\left|\E_{d\leq N}\E_{n\leq N}\theta(d)\phi_1(a_1 n+b_1d)\cdots \phi_k(a_k n+b_{k}d)\right|\leq C \|\theta\|_{U^{k}[N]} +o_{N\to \infty}(1).      
\end{align}
Since the statement of \eqref{eqq26} involves the values of the functions $\theta$ and $\phi_i$ only on $(-HN,HN)$, where $H=\max_{i\leq k}(|a_i|+|b_i|)+1$, we may assume that the functions $\theta$ and $\phi_i$ are $2HN$-periodic, and hence they can be interpreted as functions on $\mathbb{Z}_{2HN}$. We are then reduced to showing that
\begin{align*}
\left|\E_{d\in \mathbb{Z}_{2HN}}\E_{n\in \mathbb{Z}_{2HN}}\theta(d)1_{[0,N]}(d)\prod_{i=1}^k \phi_i(a_i n+db_i)1_{[0,N]}(n)\right|\leq C'\|\theta\|_{U^{k}[N]}+o_{N\to \infty}(1)    \end{align*}
for some constant $C'$, since one can then set $C \coloneqq (2H)^2 C'$. By approximating $1_{[0,N]}(n)$ with a Lipschitz function, and then further with a finite Fourier series as in \cite[Appendix C]{gt-linear}, and redefining the functions $\phi_j$, we may eliminate the factor $1_{[0,N]}(n)$. Then, making a change of variables $d=d_1+\cdots+d_k$, $n=n'-d_1b_1-\cdots-d_kb_k$, we are left with showing that
\begin{align}\label{eq8}
\left|\E_{d_1,\ldots, d_k\in \mathbb{Z}_{2HN}}\theta'(d_1+\cdots +d_k)\prod_{i=1}^k \phi_i\left(a_in'+\sum_{\ell=1}^{k}d_{\ell}(b_i-b_{\ell})\right)\right|\leq C''\|\theta\|_{U^{k}[N]}+o_{N\to \infty}(1),    
\end{align}
for all $n'\in \mathbb{Z}_{2HN}$, where $\theta'(d)\coloneqq \theta(d)1_{[0,N]}(d)$. By the Gowers-Cauchy-Schwarz inequality (see e.g., \cite[(B.7)]{gt-linear}), we have
\begin{align*}
\left|\E_{d_1,\ldots, d_k\in \mathbb{Z}_{2HN}}\theta'(d_1+\cdots +d_k)\prod_{i=1}^k\phi_i'(L_i(d_1,\ldots, d_k))\right|\leq \|\theta'\|_{U^{k}(\mathbb{Z}_{2HN})}
\end{align*}
for any functions $\theta'$ and $\phi_i'$ bounded by $1$ in modulus and any linear forms $L_i:\mathbb{Z}_{2HN}^k\to \mathbb{Z}_{2HN}$, with $L_i$ independent of the $i$th coordinate. Applying this to the left-hand side of \eqref{eq8}, where each term involving $\phi_i$ is independent of the variable $d_i$, we see that
\begin{align*}
 \left|\E_{d_1,\ldots, d_k\in \mathbb{Z}_{2HN}}\theta'(d_1+\cdots +d_k)\prod_{i=1}^k \phi_i\left(a_in'+\sum_{\ell=1}^{k}d_{\ell}(b_i-b_{\ell})\right)\right|\leq \|\theta'\|_{U^{k}(\mathbb{Z}_{2HN})}.   
\end{align*}
Then, by noting that
\begin{align*}
\|\theta(n)1_{[0,N]}(n)\|_{U^{k}(\mathbb{Z}_{2HN})}=\|\theta\|_{U^k[N]} \cdot \|1_{[0,N]}\|_{U^{k}(\mathbb{Z}_{2HN})}\leq \|\theta\|_{U^k[N]},   
\end{align*}
the lemma follows.

\end{proof}

Next, we need control on the Gowers norms for the primes.

\begin{lemma}[Gowers uniformity of the primes]\label{le2} Let $k\in \mathbb{N}$, and let $w\in \mathbb{N}$ be a large parameter. Further, let $W=\prod_{p\leq w}p$, and let $b\in [1,W]$ be coprime to $W$. Then for any $N$ large enough in terms of $w$, the $W$-tricked von Mangoldt function
\begin{align}\label{eq25}
\Lambda_{b,W}(n)\coloneqq \frac{\varphi(W)}{W}\Lambda(Wn+b)    
\end{align}
enjoys the Gowers uniformity bound
\begin{align*}
\|\Lambda_{b,W}-1\|_{U^{k+1}[N]}=o_{w\to \infty}(1).    
\end{align*}
\end{lemma}

\begin{proof} This was proven in \cite{gt-linear}, subject to conjectures that were later verified in \cite{gt-mobius}, \cite{gtz}.
\end{proof}

We now prove Theorem \ref{compar}.  Let $k,a_1,\dots,a_k,b_1,\dots,b_k, \eps,w,H_-,H_+,x,W,a,m$ be as in that theorem.  Because $\Lambda(p) = \log(2^m) + O(1)$ when $p$ is a prime with $2^m < p \leq 2^{m+1}$, and $\Lambda$ is non-zero for only $O( 2^{2m/3})$ (say) other integers in the interval $(2^m,2^{m+1}]$, we have
$$ \E^{\log}_{2^m < p \leq 2^{m+1}} f_x(ap) = \E^{\log}_{2^m < d \leq 2^{m+1}} f_x(an) \Lambda(d) + O(\eps),$$
since $m$ is assumed to be sufficiently large depending on $\eps$.
The contribution to the right-hand side of those $d$ that share a common factor with $W$ is negligible (as $\Lambda(d)$ will then vanish unless $n$ is a power of a prime less than or equal to $w$), thus
$$ \E^{\log}_{2^m < p \leq 2^{m+1}} f_x(ap) = \frac{W}{\phi(W)} \E^{\log}_{2^m < d \leq 2^{m+1}: (d,W)=1} f_x(ad) \Lambda(d) + O(\eps).$$
It therefore suffices to show that
$$
\E^{\log}_{2^m < d \leq 2^{m+1}: (d,W)=1} f_x(ad) (\frac{W}{\phi(W)} \Lambda(d)-1) \ll \eps.
$$
Partitioning into residue classes modulo $W$ and using \eqref{eq25}, it suffices to show that
$$
\E^{\log}_{2^m/W < d \leq 2^{m+1}/W} f_x(a(Wd+b)) (\Lambda_{b,W}(d)-1) \ll \eps
$$
whenever $1 \leq b \leq W$ is coprime to $W$.

Fix $b$.  By summation by parts, it will suffice to show that
$$
\E_{d \leq H} f_x(a(Wd+b)) (\Lambda_{b,W}(d)-1) \ll \eps
$$
whenever $2^m/W \leq H \leq 2^{m+1}/W$.  From \eqref{fax}, and replacing the average $n \leq x$ with the average $x^\eps < n \leq x$, we have
$$ f_x(a(Wd+b)) = \E_{x^\eps < n \leq x}^{\log}\lambda( a_1 n + W a b_1 d + abb_1) \dots \lambda( a_k n + W a b_k d + abb_k )  + O(\eps),$$
so it suffices to show that
\begin{equation}\label{targ}
\E_{d \leq H} \E_{x^\eps < n \leq x}^{\log} (\Lambda_{b,W}(d)-1) \lambda( a_1 n + W a b_1 d + abb_1) \dots \lambda( a_k n + W a b_k d + abb_k ) \ll \eps. 
\end{equation}
The quantity $x$ (or $x^\eps$) is large compared with $aHW$.  Thus we can shift $n$ by any quantity $1 \leq n' \leq aHW$ without affecting the above average by more than $O(\eps)$.  Performing this shift and then averaging in $n'$, the left-hand side of \eqref{targ} may be written as
$$
\E_{x^\eps < n \leq x}^{\log} \E_{d \leq H} \E_{n' \leq aHW} (\Lambda_{b,W}(d)-1) \lambda( a_1 n' + W a b_1 d + a_1 n + abb_1) \dots \lambda( a_k n' + W a b_k d + a_1 n + abb_k ) + O(\eps). $$
Applying Lemma \ref{le1} with $N$ replaced by $aHW$, $W$ replaced by $aW$, $n$ replaced by $n'$, and the $r_j$ replaced by $a_j n + abb_j$ for $1,\dots,k$, we can bound this as
$$
O( \|\Lambda_{b,W}-1\|_{U^k[H]} ) + o_{H\to \infty}(1) + O(\eps),$$
but by Lemma \ref{le2} this is $O(\eps)$ as required.

\end{document}